\documentclass[11pt]{article}
\usepackage[utf8]{inputenc}
\usepackage{amsmath,amsthm,amssymb}
\usepackage{tikz-cd} 
\usepackage{mathtools} 
\usepackage{textcomp} 
\usepackage{bbm} 
\usepackage{proof} 
\usepackage[tiny]{titlesec} 
\usepackage{hyperref}
\hypersetup{
	colorlinks=true,
	linkcolor=blue,
}

\usepackage{cleveref} 

\newtheorem{theorem}{Theorem}[section] 

\newtheorem{lemma}[theorem]{Lemma}
\newtheorem{proposition}[theorem]{Proposition}
\newtheorem{corollary}[theorem]{Corollary}
\newtheorem{observation}[theorem]{Observation}

\theoremstyle{definition}
\newtheorem{definition}[theorem]{Definition}

\theoremstyle{remark}
\newtheorem{remark}[theorem]{Remark}

\newcommand{\set}{\mathbf{Set}} 
\newcommand{\cat}{\mathbf{Cat}} 
\newcommand{\sset}{\mathbf{sSet}} 
\newcommand{\bisim}{\mathbf{ssSet}} 

\newcommand{\catB}{\mathcal{B}} 
\newcommand{\catC}{\mathcal{C}} 
\newcommand{\catD}{\mathcal{D}} 
\newcommand{\catM}{\mathcal{M}} 

\newcommand{\catE}{\mathcal{E}} 

 \newcommand{\refl}{\textup{\textsf{refl}}} 

\newcommand{\universetyp}{\mathcal{U}} 
\newcommand{\pointtyp}{\textup{\textbf{1}}} 

\newcommand{\isprop}{\textup{\textsf{isProp}}} 
\newcommand{\iscontr}{\textup{\textsf{isContr}}} 
\newcommand{\isequiv}{\textsf{\textup{IsEquiv}}} 

\newcommand{\simplext}{\Delta^1} 
\newcommand{\simplextt}{\Delta^2} 

\newcommand{\horn}{\Lambda_1^2} 

\newcommand{\homtyp}{\textup{\textsf{hom}}} 

\newcommand{\issegal}{\textup{\textsf{isSegal}}} 
\newcommand{\isinner}{\textup{\textsf{isInner}}} 
\newcommand{\isrezk}{\textup{\textsf{isRezk}}} 
\newcommand{\isisoinner}{\textup{\textsf{isIsoinner}}} 

\newcommand{\relfuntyp}{\textup{\textsf{Fun}}_{/B}} 
\newcommand{\iscomp}{\textup{\textsf{isCompletion}}}

\newcommand{\ie}{\textit{i.e. }}

\DeclareMathOperator{\op}{op}

\title{Exponentiable functors between synthetic $\infty$-categories}
\date{}
\author{C\'esar Bardomiano Mart\'inez \\
	Department of Mathematics and Statistics\\
	University of Ottawa \\
	\textit{cbard035@uottawa.ca}}

\begin{document}
      
	\maketitle
	
	\begin{abstract}
          We study exponentiable functors in the context of synthetic $\infty$-categories. We do this within the framework of simplicial Homotopy Type Theory of Riehl and Shulman. Our main result characterizes exponentiable functors. In order to achieve this, we explore Segal type completions. Moreover, we verify that our result is semantically sound.
	\end{abstract}
	

\section{Introduction}

\subsection{Synthetic $\infty$-category theory}

A proposal for a synthetic theory $(\infty,1)$-categories using Homotopy Type Theory appears in the seminal work of Riehl and Shulman \cite{riehlshulman}, called simplicial Homotopy Type Theory or sHoTT for short. They define \textit{Segal} and \textit{Rezk types}, which play the role of $(\infty,1)$-precategories and $(\infty,1)$-categories. The paper develops categorical properties of said types. Also, it studies discrete fibrations and adjunctions. Further work \cite{buchwein} and \cite{ces22} present (co)cartesian fibration and (co)limits, respectively.

The standard semantics of sHoTT is the category of bisimplicial sets $\bisim$ with the Reedy model structure. \cite{riehlshulman} shows that Segal types correspond to Segal spaces, and Rezk types to complete Segal spaces.
Furthermore, the main result in \cite{shulman2019all} implies that if $\catE$ is a Grothendieck $(\infty,1)$-topos, which is in particular a model of Homotopy Type Theory, then we can produce a model of sHoTT in the (internal) category of simplicial presheaves $\catE^{\Delta^{\op}}$ \cite{weinberger2022strict}.

This is the general framework of synthetic category theory in which our work takes place. We study an important class of functors, the \textit{exponentiable} ones, which we introduce shortly after.
This is also a continuation of \cite{ces22}, which started as an exploration on how far we can go in the development of synthetic $(\infty,1)$-category theory without enhancing sHoTT.

\subsection{Exponentiable functors}

An exponential object, or more generally, exponentiable map, can be defined in multiple ways. If $ \catC $ is category with binary products, we say a map $ f:E \to B $ in $ \catC $ is exponentiable if the pullbacks along $f$ exists and the functor $ f^*:\catC/B \to \catC/E $ has a right adjoint $ \prod_f $, so it induces an adjoint triple
\[\begin{tikzcd}
	{\catC/B} && {\catC/E}
	\arrow[""{name=0, anchor=center, inner sep=0}, "{f^*}"{description}, from=1-1, to=1-3]
	\arrow[""{name=1, anchor=center, inner sep=0}, "{\sum_f}"{description}, bend right=50, from=1-3, to=1-1]
	\arrow[""{name=2, anchor=center, inner sep=0}, "{\prod_f}"{description}, bend left=50, from=1-3, to=1-1]
	\arrow["\bot"{description}, draw=none, from=0, to=1]
	\arrow["\bot"{description}, draw=none, from=2, to=0]
\end{tikzcd}\]
where $ \sum_f $ is given by composition with $ f $. More generally, a \textbf{locally cartesian closed category} is a category in which every map is exponentiable. Exponentiable maps in the category of small categories $ \cat $ are also known as Conduch\'e fibrations. The literature on the topic is extensive and they appear for example in \cite{conduche}. In the context of $ \infty $-categories exponentiable functors have been studied in \cite[Lemma 5.16]{ayala2018factorization} and \cite[Appendix B.3]{lurie2}. This result is in the same spirit as \Cref{conducheoriginal} below.

Let us recall the case for categories, for which we first introduce some notation. Given $ f:\catE \to \catB $ a functor and $ a \in \catB $, we denote its fiber as $ \catE_a $. The category $ \catE_a $ has objects $ e\in \catE $ and morphisms $ k:e_1 \to e_2 \in \catE$ such that $ f(e)=a $ and $ f(k)=Id_a $. If $ u:a \to b $ in $ \catB $ and $ x\in \catE_a, \, y\in \catE_b $, then the set of arrows in $ \catE $ over $ u $ with source $ x $ and target $ y $ is denoted as $ hom_\catE^u(x,y)$, whence if $ j\in hom_\catE^u(x,y) $, then $ f(j)=u $. Therefore, we can define a profunctor $ hom_\catE:\catE_a\times \catE_b^{op} \to \set $ in the obvious way. The following statement characterizes exponentiable functors between categories and it is due to Conduch\'e \cite{conduche} and Giraud \cite{Giraud1964},

\begin{theorem}\label{conducheoriginal}
	For a functor $ f:\catE \to \catB $, the following conditions are equivalent:
	\begin{enumerate}
		\item The functor $ f: \catE\to \catB $ is exponentiable,
		\item For all $a,\,b,\,c\in \catB,\, u\in hom_\catB(a,b),\, v\in hom_\catB(b,c),\,x\in \catE_a,\,z\in \catE_c$, the induced map
		\[\left(\int_{}^{y\in \catE_b}hom_\catE^u(x,y)\times hom_\catE^v(y,z)\right) \to hom_\catE^{v\circ u}(x,z)\]
		is an isomorphism.
	\end{enumerate}
\end{theorem}

The result we prove in \Cref{conduche} can be seen as analogous to the previous statement, while partially recovering the result in \cite{ayala2018factorization} which is the $ \infty $-categorical statement of \Cref{conducheoriginal}. Partial results about exponentiable functors also appear in \cite[Appendix B.3]{lurie2}. Condition 2 from \Cref{conducheoriginal} states that the composition of the $ hom$-profunctors is given by the coend formula where the map is induced by composition. This is also what \cite{ayala2018factorization} prove in their result, which they do using the language of \textit{correspondences}. In \Cref{profunctors} we explain how this is reflected in our \Cref{conduche}.

\subsection{Outline}

To define \textit{exponentiable functors}, throughout \Cref{universalsegal}, we study \textit{Segal type completions}. This notion is essential to correctly formulate Condition 5 in \Cref{conduche}. This is exactly what we should think of as the composition of profunctors. In \Cref{conduchesection} we present \Cref{conduche}, which is the characterization of \textit{exponentiable functors} between Segal types. We then specialize this result to Rezk types in \Cref{conduchecorollary}.

Finally, in \Cref{semantics} we verify that our definitions are consistent with the semantics. The procedure we follow is first interpret our type-theoretic definition in the standard semantics, bisimplicial sets, and then verify that the resulting statement is equivalent to the existing definition. We do this for the Segal type completion in \Cref{semanticscomparisonus}. Finally, in \Cref{exponentialsemantics} we verify that our definition of exponentiable functor is semantically sound. 

\begin{par}
	\textbf{Acknowledgement.} The author wants to thank his PhD supervisor for his insights, comments and suggestions for the develoment of this work.\\
	The author also acknowledges the support of the Natural Sciences and Engineering Research Council of Canada (NSERC), under the grant, reference number RGPIN-2020-06779, awarded to Simon Henry. The author is grateful for the support granted by the Department of Mathematics and Statistics of the University of Ottawa.
\end{par}



\section{The Segal type completion} \label{universalsegal}

The notion we study in this section is essential to correctly formulate condition 5 in \Cref{conduche}. We start by establishing some basic definitions and mention some properties.

\begin{definition}
	A \textbf{Segal type completion} for a type $ A $ consist of a Segal type $S$ and a map $ \iota:A \to S $ such that for any Segal type $ X $ the map
	
	\begin{equation} \label{isuniversal}
		\iota^* :\equiv \_\circ\iota:(S\to X)\to (A\to X)
	\end{equation}
is an equivalence \ie if
	\[
	\iscomp^A(S,\iota):\equiv  \issegal(S)\times\left(\prod_{X:\universetyp}\issegal(X)\to\isequiv(\iota^*)\right).
	\]
\end{definition}

It is immediate to see that whenever $ (S,\iota) $ exists then it is unique up to equivalence.

\begin{definition}
	Let $ A $ a type. We define
	\[
	\textsf{\textup{Completion}}(A):\equiv  \sum_{S:\universetyp}\sum_{\iota:A\to S}\iscomp^A(S,\iota).
	\]
\end{definition}

The following result is a direct consequence of the Univalence Axiom:
\[\isprop(\textsf{\textup{Completion}}(A)).\]
We call the Segal type in this proposition the \textbf{Segal type completion} of the type $ A $. In the Segal space model structure on $ \bisim $ this corresponds to the fibrant replacement of a Reedy fibrant bisimplicial set by a Segal space, see \Cref{completionfibrant}. We can define in the same way the \textbf{Rezk type completion} of a type. The results that we get for Segal type completions are also true for Rezk type completions. We formally state the relative version of Rezk completions in \Cref{rezktypecompletion}.

 It is more convenient to work with Segal type completions at an informal level, a trade-off between clarity and formality. We favor this approach as it is more intuitive and it reflects the nature the semantical counterpart of these types. The equivalence (\ref{isuniversal}) tells us that the fibers are contractible; for any $\psi:A\to X$ we have
\[\iscontr\left(\sum_{\varphi:S\to X}\varphi\circ \iota=\psi\right).\]
Unfolding the above means that for any $\psi:A\to X$, where $X$ is a Segal type, there exists a unique $\varphi:S\to X$ such that $\varphi\circ \iota=\psi$. We can put this pictorially by saying that any $\psi:A\to X$ factors uniquely through $\iota$ as in the diagram below:
\[\begin{tikzcd}
	A & S \\
	& X.
	\arrow["\iota", from=1-1, to=1-2]
	\arrow["\psi"', from=1-1, to=2-2]
	\arrow["\varphi", dashed, from=1-2, to=2-2]
\end{tikzcd}\]
We will often refer only to the Segal space $S$ and assume that the map $\iota:A \to S$ is given and available for use. By uniqueness in the above we simply mean that the type \[ \iscontr\left(\sum_{\varphi:S \to X} \psi \sim \varphi \circ \iota \right) \]
is inhabited. The homotopy in the center of contraction is omitted most of the time and we only make reference to the function. We carry this convention forward below.

We can also consider a relative version of this universal property. For the time being fix a Segal type $B$.  We use the following notation:
\[
\universetyp /B :\equiv \sum_{S:\universetyp}S\to B.
\]
Furthermore, we will refer to an element $(S, \phi)$ of this type by leaving implicit the type $S$ and mentioning that we have a map of type $S \to B$. 

Recall from \cite{buchwein} that the \textbf{relative function type} between functions $\pi:A\to B$ and $\xi:E\to B$ is given by the pullback diagram:
\[\begin{tikzcd}
	{\relfuntyp(A,E)} & {E^A} \\
	{\mathbf{1}} & {B^A.}
	\arrow["{\xi^A}", from=1-2, to=2-2]
	\arrow["\pi"', from=2-1, to=2-2]
	\arrow[from=1-1, to=2-1]
	\arrow[from=1-1, to=1-2]
	\arrow["\lrcorner"{anchor=center, pos=0.125}, draw=none, from=1-1, to=2-2]
\end{tikzcd}\]
Note that if we assume further that $E$ is a Segal type, and $A$ is a type or shape, then $\relfuntyp(A,E)$ is a Segal type. An element $\iota: \relfuntyp(A,E)$ is a function making the following diagram commute:

\[\begin{tikzcd}
	A & E \\
	& B.
	\arrow["\pi"', from=1-1, to=2-2]
	\arrow["\xi", from=1-2, to=2-2]
	\arrow["\iota", from=1-1, to=1-2]
\end{tikzcd}\]
Thus, we call the elements of $\relfuntyp(A,E)$ \textbf{functions over} $B$.

\begin{definition} \label{segaltypecompletion}
	Let $ A\to B $ a type over $ B $. A \textbf{relative Segal type completion} for $ A \to B $ consist of a Segal type over $B$, $S \to B$ and a map $ \iota:\relfuntyp(A,S) $ such that for any Segal type $X$ over $B$ the map
	
	\begin{equation} \label{isuniversal_rel}
		\iota_{/B}^*:\relfuntyp(S,X)\to \relfuntyp(A,X)
	\end{equation}
	
	is an equivalence \ie if
	\[
	\iscomp_{/B}^A(S,\iota):\equiv  \issegal(S)\times \left( \prod_{X:\universetyp/B}\issegal(X)\to\isequiv(\iota_{/B}^*)\right).
	\]
\end{definition}

It is immediate to see that whenever $ S\to B $ exists then it is unique up to equivalence. Following \Cref{segaltypecompletion} we state:
\begin{definition} \label{rezktypecompletion}
	Let $ A\to B $ a type over $ B $. A \textbf{relative Rezk type completion} for $ A \to B $ consist of a Rezk type over $B$, $R \to B$ and a map $ \iota:\relfuntyp(A,R) $ such that for any Rezk type $X$ over $B$ the map
	
	\begin{equation} \label{isuniversal_rel}
		\iota_{/B}^*:\relfuntyp(R,X)\to \relfuntyp(A,X)
	\end{equation}
	
	is an equivalence \ie if
	\[
	\textsf{\textup{isRCompletion}}_{/B}^A(S,\iota):\equiv  \isrezk(R)\times \left( \prod_{X:\universetyp/B}\isrezk(X)\to\isequiv(\iota_{/B}^*)\right).
	\]
      \end{definition}

      It is again immediate to see that whenever $ R\to B $ exists then it is unique up to equivalence.

\begin{definition}
	Let $ A \to B$ a type over $ B $. We define
	\[
	\textsf{\textup{Completion}}_{/B}(A):\equiv  \sum_{S:\universetyp/B}\sum_{ \iota:\relfuntyp(A,S) }\iscomp_{/B}^A(S,\iota).
	\]
\end{definition}

We will say often that $S$ is a Segal type completion relative to the type $B$ leaving implicit the map $\xi:S\to B$, and that $\iota:\relfuntyp(A,S)$ exists. The equivalence (\ref{isuniversal_rel}) tells us that the fibers are contractible, for any $\psi:\relfuntyp(A,X)$ we have
\[\iscontr\left(\sum_{\varphi:\relfuntyp(S,X)}\varphi\circ \iota=\psi\right).\]
Just as we did before, unfolding the above means that for any $\psi:\relfuntyp(A,X)$, where $\delta:X\to B$ is a map between Segal types, there exists a unique $\varphi:\relfuntyp(S,X)$ such that $\varphi\circ \iota=\psi$. The picture for this situation is the commutative diagram below:
\[\begin{tikzcd}
	A && X \\
	& S \\
	& B.
	\arrow["\pi"', from=1-1, to=3-2]
	\arrow["\xi"{description}, from=2-2, to=3-2]
	\arrow["\iota", from=1-1, to=2-2]
	\arrow["\psi", from=1-1, to=1-3]
	\arrow["\delta", from=1-3, to=3-2]
	\arrow["\varphi", dashed, from=2-2, to=1-3]
\end{tikzcd}\]

For the next section it will be useful to know that for a type $A$ its associated Segal type completion $S$ is also universal relative to any Segal type $B$, and vice versa. Informally, the categorical interpretation we give to this is that having the Segal type completion over the single point Segal space is equivalent to having it over any slice (by a Segal space). This is the content of:

\begin{proposition} \label{relunivsegaltype}
	Let $A$ be any type and $ B $ a Segal type. Assume further we have a commutative diagram
	\[\begin{tikzcd}
		A && S \\
		& B.
		\arrow["\pi"', from=1-1, to=2-2]
		\arrow["\xi", from=1-3, to=2-2]
		\arrow["\iota", from=1-1, to=1-3]
	\end{tikzcd}\]
	where $ S $ is a Segal type. Then \[\iscomp(S,\iota) \simeq \iscomp_{/B}^A(S,\iota).\]
\end{proposition}
\begin{proof}
	Since $ S $ is a Segal type it is enough to show that \[ \left(\prod_{X:\universetyp}\issegal(X)\to\isequiv(\iota^*)\right) \simeq \left(\prod_{X:\universetyp/B}\issegal(X)\to\isequiv(\iota_{/B}^*)\right).\]
	It is possible to construct a function \[\varphi:\left(\prod_{X:\universetyp}\issegal(X)\to\isequiv(\iota^*)\right) \to \left(\prod_{X:\universetyp/B}\issegal(X)\to\isequiv(\iota_{/B}^*)\right)\]
	after some preliminary observations we now make.
        
	Let $\delta:X\to B$, where $ X $ a Segal type, and $j:\relfuntyp(A,X)$. We wish to see that
	\[ \iscontr\left(\textsf{fib}_{\iota_{/B}^*}(j)\right). \]
	By assumption we have the equivalence $ \iota^*:X^S \to X^A $, this give us unique functions $h:S\to B$ and $g:S\to X$ making the following diagrams commutative:
	\[\begin{tikzcd}
		A & S & A & S \\
		& X && B.
		\arrow["j"', from=1-1, to=2-2]
		\arrow["\iota", from=1-1, to=1-2]
		\arrow["g", dashed, from=1-2, to=2-2]
		\arrow["\pi"', from=1-3, to=2-4]
		\arrow["\iota", from=1-3, to=1-4]
		\arrow["h", dashed, from=1-4, to=2-4]
	\end{tikzcd}\]
	Since $j:\relfuntyp(A,X)$ then by uniqueness of $h:S\to B$ we must have $h=\delta \circ g$. It is now clear that $g$ must be unique with this property. All together fit in the diagram:
	\[\begin{tikzcd}
		A && S \\
		& X \\
		& B.
		\arrow["\pi"', from=1-1, to=3-2]
		\arrow["h", from=1-3, to=3-2]
		\arrow["\iota", from=1-1, to=1-3]
		\arrow["j", from=1-1, to=2-2]
		\arrow["\delta"{description}, from=2-2, to=3-2]
		\arrow["g"{description}, dashed, from=1-3, to=2-2]
	\end{tikzcd}\]
      If we were to unfold $ \isequiv(\iota_{B}^*) $ and $ \isequiv(\iota_{/B}^*) $ we would be able to give an explicit formula for $ \varphi $, then the above shows that $ \varphi $ is well-defined. We will not do this since it does not gives any more clarity on the result.
      
    Likewise, it is possible to construct
    \[\psi:\left(\prod_{X:\universetyp/B}\issegal(X)\to\isequiv(\iota_{/B}^*)\right) \to \left(\prod_{X:\universetyp}\issegal(X)\to\isequiv(\iota^*)\right).\]
     Let $j:A\to X$ be a function where $X$ is a Segal type, now we want to observe that \[\iscontr\left( \textsf{fib}_{i^*}(j) \right). \]
     Consider $j\times \mathrm{Id}_B:A\times B\to X\times B$. We construct the commutative diagram:
	\[\begin{tikzcd}
		A && {A\times B} & {} & {X\times B} \\
		&& B.
		\arrow["\pi"', from=1-1, to=2-3]
		\arrow["{p_2}"{description}, from=1-3, to=2-3]
		\arrow["{(\mathrm{Id}_A,\pi)}", from=1-1, to=1-3]
		\arrow["{j\times \mathrm{Id}_B}", from=1-3, to=1-5]
		\arrow["{p_2}", from=1-5, to=2-3]
	\end{tikzcd}\]
	This implies that $(j,\pi):\relfuntyp(A,X\times B)$. Since $S$ is the Segal type completion of $ A $ relative to $B$ there exists a unique function $f:S\to X\times B$ making the following diagram commutative:
	\[\begin{tikzcd}
		A && {X\times B} \\
		& S \\
		& B.
		\arrow["\pi"', from=1-1, to=3-2]
		\arrow["{p_2}", from=1-3, to=3-2]
		\arrow["{(j,\pi)}"{description}, from=1-1, to=1-3]
		\arrow["\iota", from=1-1, to=2-2]
		\arrow["\xi"{description}, from=2-2, to=3-2]
		\arrow["f"{description}, dashed, from=2-2, to=1-3]
	\end{tikzcd}\]
	From this we obtain:
	\[\begin{tikzcd}
		A & S \\
		& X.
		\arrow["j"', from=1-1, to=2-2]
		\arrow["\iota", from=1-1, to=1-2]
		\arrow["{p_1\circ f}", dashed, from=1-2, to=2-2]
	\end{tikzcd}\]
	To show uniqueness, if we had a map $g:S\to X$ fitting in the triangle above then we certainly get:
	\[\begin{tikzcd}
		A && {X\times B} \\
		& S \\
		& B.
		\arrow["\pi"', from=1-1, to=3-2]
		\arrow["{p_2}", from=1-3, to=3-2]
		\arrow["{(j,\pi)}"{description}, from=1-1, to=1-3]
		\arrow["\iota", from=1-1, to=2-2]
		\arrow["\xi"{description}, from=2-2, to=3-2]
		\arrow["{(g,\xi)}"{description}, from=2-2, to=1-3]
	\end{tikzcd}\]
	Uniqueness implies that $f=(g,\xi)$, from which we finally conclude $p_1\circ f= g$.\\
	After obtaining $ \varphi $ and $ \psi $ we can then show that $ \varphi \circ \psi $ and $ \psi \circ \varphi $ are homotopic the corresponding identities. This follows using the universality of each completion.
\end{proof}

\section{Conduch\'e's theorem} \label{conduchesection}

The next theorem characterizes and at the same time allows to define exponentiable functors between Segal types. As we mentioned in the introduction, this result is analogous to the one given by Ayala, Francis and Rozenblyum in \cite[Lemma 5.16]{ayala2018factorization} for quasi-categories. Its proof is the main focus of this section.

Recall from \cite{buchwein} that a type family $ Q:B\to \universetyp $ is an \textbf{inner family} if
\[
 \isinner(Q):\equiv \prod_{\alpha:\simplextt \to B }\prod_{\delta:\prod_{t:\horn}Q(\alpha(i(t)))} \iscontr\left(\left<\prod_{t:\simplextt}Q(\alpha(t))\middle|_{\delta}^{\horn} \right> \right).
\]

The significance of inner families can be understood via the following result which can be found in \cite[Proposition 4.1.5 and 4.1.6]{buchwein}:

\begin{proposition} \label{mapsbetweensegal}
	Let $ Q:B \to \universetyp $ be a type family over a Segal type $ B $. Then $ \issegal\left(\sum_{b:B}Q(b)\right) $ if and only if $ \isinner(Q) $.
\end{proposition}

Let us explain the proposition above. Firstly, given a type $A$ we can think of it as a type family over $\pointtyp$ \ie $\lambda *.A:\pointtyp \to \universetyp$. This give us a function $\pi:A \to \pointtyp$. Then the proposition above says that $A$ is a Segal type if and only the diagram
\[\begin{tikzcd}
	{\Lambda_1^2} & A \\
	{\Delta^2} & \pointtyp
	\arrow[from=1-1, to=1-2]
	\arrow[from=1-1, to=2-1]
	\arrow["\pi", from=1-2, to=2-2]
	\arrow[dashed, from=2-1, to=1-2]
	\arrow[from=2-1, to=2-2]
      \end{tikzcd}\]
    has a diagonal filler which is unique up to homotopy. \Cref{mapsbetweensegal} establishes a relative version of this condition. If $Q:B \to \universetyp$ is a type family over the Segal type $B$ and $\pi:\tilde{Q} \to B$ is the cannonical projection from its total type, then \Cref{mapsbetweensegal} can be rephrased by saying that the diagram
    \[\begin{tikzcd}
	{\Lambda_1^2} & \tilde{Q} \\
	{\Delta^2} & B
	\arrow["\delta",from=1-1, to=1-2]
	\arrow[from=1-1, to=2-1]
	\arrow["\pi", from=1-2, to=2-2]
	\arrow[dashed, from=2-1, to=1-2]
	\arrow["\alpha"',from=2-1, to=2-2]
      \end{tikzcd}\]
    has a diagonal filler which is unique up to homotopy, see \cite[Observation 2.4.1]{buchwein}. We proceed to prove the main result of this section.

\begin{theorem} \label{conduche}
	Let $E:B \to \universetyp$ an inner family over a Segal type $ B $, the following are equivalent:
	
	\begin{enumerate}
		\item For any inner family $P:B \to \universetyp$,
		\[\issegal\left(\sum_{b:B}\left(E(b)\to P(b)\right)\right).\]
		
		\item For any inner family $P:B \to \universetyp $, the type family $Q :\equiv  \lambda b. (E(b)\to P(b)):B\to \universetyp$ is inner.
		
		\item For any inner family $P:E\to \universetyp $, the type family $Q :\equiv  \lambda b. \prod_{e:E(b)}P(e) :B\to \universetyp$ is inner.
		
		\item For any Segal type $X$, the type family $Q :\equiv  \lambda b.E(b)\to X:B\to \universetyp$ is inner.
		
		\item \label{coduche5} For any map $ \alpha:\simplextt \to B $, together with the inclusion $ i:\horn \to \simplextt $. Let $ F_1:\equiv  \sum_{t:\simplextt} E(\alpha(t)) $ and $ F_2:\equiv  \sum_{t:\horn} E(\alpha(i(t))) $ then
		\[
		 \textsf{\textup{isCompletion}}_{/B}^{F_2}(F_1,\iota),
		\]
		where $ \iota :\equiv  \lambda(t,e).(i(t),e) \ : F_2 \to F_1 $.
	\end{enumerate}
\end{theorem}

Before proceeding with the proof we give some motivations for the conditions of the theorem. If a map $ f:E \to B\in \catC $ is exponentiable then we have a triple adjunction
\[\begin{tikzcd}
	{\catC/B} && {\catC/E}.
	\arrow[""{name=0, anchor=center, inner sep=0}, "{f^*}"{description}, from=1-1, to=1-3]
	\arrow[""{name=1, anchor=center, inner sep=0}, "{\sum_f}"{description}, bend right=50, from=1-3, to=1-1]
	\arrow[""{name=2, anchor=center, inner sep=0}, "{\prod_f}"{description}, bend left=50, from=1-3, to=1-1]
	\arrow["\bot"{description}, draw=none, from=0, to=1]
	\arrow["\bot"{description}, draw=none, from=2, to=0]
\end{tikzcd}\]
In this situation, the internal hom in the slice $ \catC/B $ is given by $$ [f:E \to B,-]_{/B}:\equiv  \prod_f \circ f^* .$$ In general, this formula is the semantic interpretation of the type family in (2) of \Cref{conduche}, $ \lambda b. E(b)\to P(b)$. In our framework of sHoTT, the extra condition of being an inner family is a natural one since maps between Segal types are equivalent to inner families (\Cref{mapsbetweensegal}). Conditions (2)-(3) above simply express the fact that exponentiation happens over the point, over $ B $ or over a type dependent on $ B $. An explanation behind (5) can be found in \Cref{profunctors} below.

\begin{proof}
  (1)$\Leftrightarrow$(2). This is immediate from \cite[Proposition 4.1.5]{buchwein} where it is proved that the total space of a type family $P:B\to \universetyp$ is a Segal type if and only if $P$ is an inner family.\\
  $(4)\Leftrightarrow (3) \Leftrightarrow (2)$ is a classical result which can be found for example in \cite{niefield}. \\
	(1)$\Rightarrow$(5). Consider the projections $ p_2: F_2 \to \horn $, $ p_1:F_1 \to \simplextt $. Denote by $ q: F_1 \to \sum_{b:B}E(b) $ the map $ \lambda (t, e).(\alpha(t),e) $. Let $ f:(\sum_{b:B}E(b)) \to B $ the projection, thus $ \alpha\circ p_1 = f\circ q $.
	
	We prove that $ f\circ q: F_1 \to B $ is the Segal type completion relative to $ B $ of $ f\circ q \circ \iota: F_2 \to B $. We will constantly make use of the commutative diagram
\[\begin{tikzcd}
	{F_2} & {F_1} & E \\
	{\Lambda_1^2} & {\Delta^2} & B.
	\arrow["\iota", from=1-1, to=1-2]
	\arrow["{p_2}"', from=1-1, to=2-1]
	\arrow["q", from=1-2, to=1-3]
	\arrow["{p_1}"', two heads, from=1-2, to=2-2]
	\arrow["f", two heads, from=1-3, to=2-3]
	\arrow["i"', from=2-1, to=2-2]
	\arrow["\alpha"', from=2-2, to=2-3]
      \end{tikzcd}\]
    In other words, we use that $ \alpha \circ p_1=f\circ q $ and $ \alpha \circ i \circ p_2 = f\circ q \circ \iota $.
    
	If there is a map $ k:X \to B $ with $ X $ a Segal type and $ \psi: F_2 \to X $ over $ B $, then we can assume we have the canonical projection $\pi:\left(\sum_{b:B}P(b)\right)\to B$ where $ P:B\to \universetyp $ is an inner family given by $\lambda b.\textsf{fib}_k(b)$ \ie by taking the fiber of $k$ over each $b:B$. This gives us the map $ \psi : F_2 \to \sum_{b:B}P(b)$ such that $ \alpha \circ i \circ p_2 =\pi \circ \psi $. Our first goal is to construct
	\[\nu :\prod_{t:\horn}\left(E(\alpha(i(t)))\to P(\alpha(i(t)))\right).\]
	We can assume that $\psi(t,e)\equiv (b_t,e_t):\sum_{b:B}P(b)$. For each $(t,e):F_2$ there is a path $ p_t: \alpha(i(t))=b_t $ given by $ \alpha \circ i \circ p_2 =\pi \circ \psi $. Then we can consider the transport map \[p_t^*:P(\alpha (i(t)))\to P(b_t)\] together with its inverse \[(p_t^{-1})^*: P(b_t) \to P(\alpha (i(t))).\]
	Define $\nu:\equiv \lambda t.\lambda e. (p_t^{-1})^*(e_t)$. This gives us the lifting problem:
	\[\begin{tikzcd}
		\horn & {\sum_{b:B}(E(b)\to P(b))} \\
		\simplextt & B.
		\arrow["i"', from=1-1, to=2-1]
		\arrow["{\pi'}", from=1-2, to=2-2]
		\arrow["\alpha"', from=2-1, to=2-2]
		\arrow["\nu", from=1-1, to=1-2]
	\end{tikzcd}\]
	By assumption, there exists a unique \[\xi:\left<\prod_{t:\simplextt}(E(\alpha(t))\to P(\alpha (t)))\middle|_{\nu}^{\horn} \right>.\]
	The construction of
	\[\varphi:\left(\sum_{t:\simplextt}E(\alpha(t))\right)\to \left(\sum_{b:B}P(b) \right)\]
	is simply given by the formula $\lambda (t,e).(\alpha(t),\xi_t(e))$. It remains to show $\phi\circ \iota\sim \psi$. Consider $(t,e):\sum_{t:\simplextt}E(\alpha(t))$. Then by definition we have 
	\[\phi\circ \iota(t,e)\equiv (\alpha(i(t)),\xi_{i(t)}(e))\text{ and }\psi(t,e)\equiv (b_t,e_t).\] 
	Using the characterization of paths in the total space we provide $p_t:\alpha(i(t))=b_t$, and also there is an equality
	
	\[p_t^*(\xi_{i(t)}(e))\equiv p_t^*(\nu_t(e))\equiv p_t^*((p_t^{-1})^*(e_t))=e_t.\]
	Therefore, $\phi\circ \iota(t,e)=\psi(t,e),$ which gives us the required homotopy.
	To prove uniqueness, let
	\[\varphi':\left(\sum_{t:\simplextt}E(\alpha (t))\right)\to \left(\sum_{b:B}P(b)\right)\]
	over $B$ and a homotopy $\bar r : \varphi' \circ \iota \sim \psi $. We can assume that $\varphi'(t,e)\equiv (b_t',e_t')$. There is a homotopy $q:\alpha\circ p_1 \sim g\circ \varphi'$. For any $(t,e):\sum_{t:\simplextt}E(\alpha (t))$ we get a path $q_t:\alpha(t)=b_t'$, which gives rise to the transport map $q_t^*:P(\alpha(t))\to P(b_t')$.\\
	Similarly, if $(t,e):\sum_{t:\horn}E(\alpha (i(t)))$ then
	\[\bar{r}_t:\varphi'(\iota(t,e))\equiv (b_{i(t)}',e_t')=(b_t,e_t) \equiv \psi(t,e).\]
	This is a path in a total space, so it is given by $r_t:b_{i(t)}'=b_t$ and a path $d_t:r_t^*(e_t')=e_t$ where $r_t^*:P(b_{i(t)}')\to P(b_t)$ is the transport map.

	We first make the following observation:	
	\begin{equation}\label{transportequality}
		\prod_{(t,e):F_2}\prod_{p_t:\alpha
			(i(t))=b_t}\prod_{q_{i(t)}:\alpha(i(t))=b_{i(t)}'}\prod_{r_t:b_{i(t)}'=b_t} (q_{i(t)}^{-1})^*((r_t^{-1})^*(e_t))=(p_t^{-1})^*(e_t).   
	\end{equation}
	By induction on the paths we assume that $p_t\equiv \refl :\alpha(i(t))=\alpha(i(t))$, $q_{i(t)}\equiv\refl :\alpha(i(t))=\alpha(i(t))$ and $r_t\equiv \refl:\alpha(i(t))=\alpha(i(t))$. In this case the transport maps are identities so we can use $\refl:e_t=e_t$.
	Using the map $\varphi'$ we construct
	\[
	\xi':\left<\prod_{t:\simplextt}(E(\alpha(t))\to P(\alpha (t)))\middle|_{\nu}^{\horn} \right>
	\]
	by $\xi'\equiv \lambda t.\lambda e. (q_t^{-1})^*(e_t')$. To see this indeed gives the correct type to $\xi'$ we evaluate on $t:\horn$. Our aim is to construct for each $e:E(\alpha (i(t)))$ a path $\xi_{i(t)}'(e)=\nu_t(e)$. This is 
	\[(q_{i(t)}^{-1})^*(e_{i(t)}')=(p_t^{-1})^*(e_t).\]
	Note that this follows from (\ref{transportequality}) in combination with the path $d_t:r_t^*(e_t')=e_t$.
	The uniqueness of $\xi$ gives us the equality $\xi'=\xi$.
	
	We now show that for all $(t,e):\sum_{t:\simplextt}E(\alpha(t))$, $\varphi(t,e)=\varphi'(t,e)$, so we need $\varphi(t,e)\equiv (\alpha(t),\xi_t(e))=(b_t',e_t')\equiv \varphi'(t,e)$. First, there is a path $q_t:\alpha(t)=b_t'$. Observe that
	\[
	q_t^*(\xi_t(e))=q_t^*(\xi_t'(e))=q_t^*((q_t^{-1})^*(e_t'))=e_t'.
	\]	
	The above proves that $F_1$ is a Segal type completion for $F_2$.\\
	
	(5)$\Rightarrow$ (1). Let $P: B \to \universetyp$ an inner type family. Using the equivalence (1)$\Leftrightarrow$(2), it is enough to show that the type family \[P:\equiv  \lambda b.E(b)\to P(b):B\to \universetyp\] is an inner family. This amounts to showing that the projection map
	\[\pi:\left(\sum_{b:B}(E(b)\to P(b))\right)\to B\]
	is right orthogonal to the horn inclusion $i:\horn \to \simplextt$. Consider a lifting problem
	\[\begin{tikzcd}
		\horn & {\sum_{b:B}\big(E(b)\to P(b)\big)} \\
		\simplextt & B.
		\arrow["i"', from=1-1, to=2-1]
		\arrow["\pi", from=1-2, to=2-2]
		\arrow["\alpha"', from=2-1, to=2-2]
		\arrow[from=1-1, to=1-2]
	\end{tikzcd}\]
	This means we have a partial section $\delta:\prod_{t:\horn}\big(E(\alpha(i(t)))\to P(\alpha(i(t)))\big)$. We define the function
	\[\psi:\left(\sum_{t:\horn}E(\alpha(i(t)))\right)\to \left(\sum_{b:B}P(b)\right)\]
	 as $\psi :\equiv \lambda t.\lambda e.(\alpha(i(t)),\delta_t(e))$. We illustrate this in a commutative diagram
	\[\begin{tikzcd}
		& {\sum_{b:B}P(b)} \\
		{F_2} & {F_1} & {\sum_{b:B}E(b)} \\
		{\Lambda_1^2} & {\Delta^2} & B.
		\arrow["f", two heads, from=2-3, to=3-3]
		\arrow["\alpha"', from=3-2, to=3-3]
		\arrow["{p_1}"', two heads, from=2-2, to=3-2]
		\arrow["{p_2}"', from=2-1, to=3-1]
		\arrow["i"', from=3-1, to=3-2]
		\arrow["\iota"', from=2-1, to=2-2]
		\arrow["\psi", from=2-1, to=1-2]
		\arrow["g"{pos=0.3}, from=1-2, to=3-3]
		\arrow[from=2-2, to=2-3, crossing over]
	\end{tikzcd}\]
	Since $P$ is an inner fammily, the universal property of $F_1$ implies we can complete this diagram to a unique map $\varphi:F_1\to \sum_{b:B}P(b)$ over $B$. In what follows we can assume that $\varphi(t,e)\equiv (b_t,e_t)$. We have paths \[p_t:\alpha(t)=b_t \text{ and } \bar{q}_t:\psi(t,e)\equiv (\alpha(i(t)),\delta_t(e))=(b_{i(t)},e_{i(t)})\equiv \varphi(\iota(t,e)).\]
	The second path amounts to 
	\[
	q_t:\alpha(i(t))=b_{i(t)} \text{ and } d_t:q_t^*(\delta_t(e))=e_{i(t)}
	\]
	where again $q_t^*:P(\alpha(i(t)))\to P(b_{i(t)})$ is the transport map. The element
	\[\xi:\left<\prod_{t:\simplextt}(E(\alpha(t))\to P(\alpha (t)))\middle|_{\delta}^{\horn} \right>\]
	is given by the formula $\xi_t(e):\equiv (p_t^{-1})^*(e_t)$. We have a function in
	
	\begin{equation} \label{transportequality2}
		\prod_{(t,e):F_2}\prod_{p_{i(t)}:\alpha(i(t))=b_{i(t)}}\prod_{q_t:\alpha(i(t))=b_{i(t)}}p_{i(t)}^*=q_t^*.
	\end{equation}
	Indeed, by path induction we can assume that $p_{i(t)}\equiv \refl :\alpha(i(t))=\alpha(i(t))$ and $q_t\equiv \refl :\alpha(i(t))=\alpha(i(t))$. Moreover, in this case the transport maps are identities, therefore the claimed equality holds.
	From (\ref{transportequality2}) we get for all $(t,e):F_2$,
	\[
	\xi_{i(t)}(e)\equiv (p_{i(t)}^{-1})^*(e_{i(t)})=(q_{t}^{-1})^*(e_{i(t)})=\delta_t(e).
	\]	
	Assume an element
	\[\xi':\left<\prod_{t:\simplextt}(E(\alpha(t))\to P(\alpha (t)))\middle|_{\delta}^{\horn} \right>.\]	
	Next, the function
	\[
	\varphi':\left(\sum_{t:\simplextt}E(\alpha(t))\right)\to \left(\sum_{b:B}P(b)\right)
	\]
	is defined as $\lambda(t,e).(\alpha(t),\xi_t'(e)).$ We observe that if $(t,e):\sum_{t:\horn}E(\alpha(i(t)))$ then	
	\[
	\varphi'(\iota(t,e))\equiv \varphi'(i(t),e)\equiv (\alpha(i(t)),\xi_{i(t)}'(e))\equiv (\alpha(i(t)),\delta_t(e))\equiv \psi(t,e).
	\]
	where the middle definitional equality holds since $\xi_{i(t)}'(e)\equiv\delta_t(e)$. By uniqueness we have $\varphi=\varphi'$. This means that 
	\[\varphi(t,e)\equiv (b_t,e_t)=(\alpha(t),\xi_t'(e))\equiv \varphi'(t,e)\]
	This equality implies that for $p_t^{-1}:b_t=\alpha(t)$ we also get an equality $(p_t^{-1})^*(e_t)=\xi_t'(e)$. Therefore, $\xi_t(e)=\xi_t'(e)$, proving uniqueness of the extension $\xi$.
	This shows that the type $\sum_{b:B}(E(b)\to P(b))$ is a Segal type.
\end{proof}

\begin{definition} \label{segalexponential}
  An inner family $E:B \to \universetyp $ over a Segal type $B$ is said to be \textbf{Segal exponentiable} if it satisfies any of the equivalent conditions of \Cref{conduche}. Moreover, a function $f: E \to B$ between Segal types is \textbf{Segal exponentiable} if the family $\lambda b.\textsf{fib}_f(b): B \to \universetyp$ is Segal exponentiable, where $\textsf{fib}_f(b)$ denotes the fiber of $f$ over $b:B$.
\end{definition}

\begin{remark} \label{segalexponential-remark}
	Observe that in the semantics, Segal spaces correspond to $( \infty,1) $-precategories. Therefore, the above \Cref{conduche} refers to exponentiability of functors between $ (\infty,1) $-precategories, hence the name we suggest. Nima Rasekh pointed out to the author that, as the theorem shows, completeness does not play any role in exponentiability. This characteristic is new and does not appear in quasicategories. Certainly this aspect is not apparent in categories neither. We move swiftly to specialize our theorem to Rezk types.
\end{remark}

From \cite{buchwein} recall that a type family $ P:B\to \universetyp $ over a Segal type $ B $ is called \textbf{isoinner family} if the following proposition is true
\[
\isisoinner(P):\equiv \isinner(P)\times \prod_{b:B} \isrezk(P(b)).
\]

For our interests we can assume always the types involved are Rezk types. From \cite[Proposition 10.9]{riehlshulman} that if $ B $ is a Rezk type and $ X $ is any type or shape then $ B^X $ is also a Rezk type.

\begin{definition}
	Let $ A \to B$ a type over $ B $. We define
	\[
	\textsf{\textup{RCompletion}}_{/B}(A):\equiv  \sum_{R:\universetyp/B}\sum_{ \iota:\relfuntyp(A,R) }\textsf{\textup{isRCompletion}}_{/B}^A(R,\iota).
	\]
\end{definition}

\begin{corollary} \label{conduchecorollary}
	Let $E:B \to \universetyp$ an isoinner family over a Rezk type $ B $, the following are equivalent:
	\begin{enumerate}
		\item For any isoinner type family $P:B \to \universetyp$,
		\[\isrezk\left(\sum_{b:B}\left(E(b)\to P(b)\right)\right).\]
		
		\item For any isoinner family $P:B \to \universetyp $, the type family $Q :\equiv  \lambda b. E(b)\to P(b):B\to \universetyp$ is isoinner.
		
		\item For any isoinner family $P:E\to \universetyp $, the type family $Q :\equiv  \lambda b. \prod_{e:E(b)}P(e) :B\to \universetyp$ is isoinner.
		
		\item For any Rezk type $X$, the type family $Q :\equiv  \lambda b.E(b)\to X:B\to \universetyp$ is isoinner.
		
		\item \label{conduchecorollary5} For any map $ \alpha:\simplextt \to B $, together with the inclusion $ i:\horn \to \simplextt $. Let $ F_1:\equiv  \sum_{t:\simplextt} E(\alpha(t)) $ and $ F_2:\equiv  \sum_{t:\horn} E(\alpha(i(t))) $ then
		\[
		\textsf{\textup{isRCompletion}}_{/B}^{F_2}(F_1,\iota),
		\]
		where $ \iota :\equiv  \lambda(t,e).(i(t),e) \ : F_2 \to F_1 $.
	\end{enumerate}
\end{corollary}

\begin{proof}
  (2)$ \Rightarrow $(1). Follows from \cite[Proposition 4.2.6]{buchwein} which proves that the total space of an inner family over a Rezk type is a Rezk type. \\
  (1)$ \Rightarrow $(2). \Cref{conduche} shows that $ \isinner(Q) $. For each $ b:B $ we have $ \isrezk(E(b)\to P(b)) $ since such fiber can be obtained as a pullback from the Rezk type $ \sum_{b:B}\left(E(b)\to P(b)\right) $.\\
  (2) $\Leftrightarrow$ (3) $\Leftrightarrow$ (4) is a classical result. \\
	(5)$ \Rightarrow $(2). From \Cref{conduche} the family is inner. Since for each $ b $, $ \isrezk(P(b)) $ then each fiber $ E(b)\to P(b) $ is also a Rezk type. \\
	(1)$ \Rightarrow $(5). The proof that $ F_1 $ is the completion is the same as in \Cref{conduche}. We just need to show that it is Rezk. This again follows from \cite[Proposition 4.2.6]{buchwein} and pullback stability of Rezk types.
\end{proof}

\begin{definition}
  An isoinner family $E:B \to \universetyp $ over a Rezk type $B$ is said to be \textbf{exponentiable} if it satisfies any of the equivalent conditions of \Cref{conduchecorollary}. Moreover, a function $f: E \to B$ between Rezk types is an \textbf{exponentiable functor} if the family $\lambda b.\textsf{fib}_f(b): B \to \universetyp$ is exponentiable, where $\textsf{fib}_f(b)$ denotes the fiber of $f$ over $b:B$.
\end{definition}

The terminology ``functor'' in the definition above is justified by the functorial behaviour of functions between Segal types, see \cite[Proposition 6.1]{riehlshulman}. We have reserve the name exponentiable functor till this point in view of \Cref{segalexponential-remark}. From \Cref{conduche} and \Cref{conduchecorollary} it would seems that we have two notions of exponential functors, one for Segal types and other for Rezk types. However, both coincide when we restrict to Rezk types.

\begin{corollary}
 Let $ E: B \to \universetyp $ an isoinner type family over a Rezk type $ B $. Then $ E $ is Segal exponentiable if and only if $ E $ is exponentiable.
\end{corollary}
\begin{proof}
	We observe that Condition 2, respectively, in each result are equivalent. The forward direction is obvious, since by definition any isoinner family is in particular an inner family. Conversely, we just need to show that any fiber $ E(b) \to P(b) $ is Rezk. But this follows form \cite[Proposition 10.9]{riehlshulman} since each $ P(b) $ is Rezk.
\end{proof}


\section{The bisimplicial sets semantics of sHoTT} \label{semantics}

In this section we check that our synthetic definitions of Segal type completion and exponentiable functors are semantically correct. In \Cref{semanticscomparisonus} verify that the Segal type completion is consistent with the usual semantics. We finalize with exponentiable functors in \Cref{exponentialsemantics}.
We use the fact the semantics of sHoTT is the category of bisimplicial sets $ \bisim. $ The details of the sematics are found in \cite{riehlshulman}. We also recommend \cite[Section 6]{ces22} for a discussion.

\subsection{Segal type completion and Segal space completion} \label{semanticscomparisonus}

Here we check the soundness of a Segal type completion by comparing it with the Segal space completion defined in bisimplicial sets. Furthermore, since in general we want to consider dependent types, we need a relative version. Given a Segal space $ B $ we consider the induced model structure on the slice $ \bisim/B $. The existence of such a Segal completion is given by the fibrant replacement in the Segal model structure on $ \bisim $, and in the relative case the fibrant replacement in the slice $ \bisim/B $ (see \Cref{completionfibrant}).

Recall that for objects $ \pi:A\to B $ and $ \xi:S \to B $ in $ \bisim/B  $, the relative mapping space is denoted as $ Map_{/B}(A,S) $. This space is obtained by the pullback square

\[\begin{tikzcd}
	{Map_{/B}(A,S)} & {Map(A,S)} \\
	\Delta[0] & {Map(A,B).}
	\arrow["\pi"', from=2-1, to=2-2]
	\arrow["{\xi^A}", from=1-2, to=2-2]
	\arrow[from=1-1, to=2-1]
	\arrow[from=1-1, to=1-2]
	\arrow["\lrcorner"{description, pos=0}, draw=none, from=1-1, to=2-2]
\end{tikzcd}\]
On the other hand, the Segal space of functions between $ \pi $ and $ \xi $ is given by the following pullback square
\[\begin{tikzcd}
	{Fun_{/B}(A,S)} & {S^A} \\
	F(0) & {B^A.}
	\arrow["\pi"', from=2-1, to=2-2]
	\arrow[from=1-1, to=2-1]
	\arrow[from=1-1, to=1-2]
	\arrow["\xi^A", from=1-2, to=2-2]
	\arrow["\lrcorner"{description, pos=0}, draw=none, from=1-1, to=2-2]
\end{tikzcd}\]
Observe that $ Fun_{/B}(A,S)_0=Map_{/B}(A,S) $.

\begin{definition}
	Let $ \pi:A\to B $ and $\xi: S \to B $ in $ \bisim/B  $. Assume further that $ S $ is a Segal space and there is a map $ \iota: A\to S\in \bisim/B.$ We say that $ S $ is a \textbf{Segal space completion relative} to $ B $ for $ A $ if for any Segal space over $ B $, $ \delta:X \to B $, the induced map
	\[
	Map_{/B}(S,X)\to Map_{/B}(A,X)
	\]
	is an equivalence of spaces.
\end{definition}

This definition is the generalization of the completion of a Segal space into a complete one as defined by Rezk in \cite{rezk}. A related notion of completion of a precategory into a category in the context of Homotopy Type Theory due to Ahrens, Kapulkin and Shulman appears in \cite{ahrens2015univalent}, where the authors use the suggestive name ``Rezk completion.''

\begin{observation} \label{tuniversaltouniversal}
	Note that the interpretation of the map (\ref{isuniversal_rel}) from \Cref{segaltypecompletion} into bisimplicial sets gives us an equivalence between Segal space
	\[Fun_{/B}(S,X)\to Fun_{/B}(A,X), \]
	which is just to say that we have a level-wise equivalence of spaces. Thus, for any function $ \pi:A\to B $, with $ B $ a Segal type, the relative Segal type completion $\xi:S\to B$ for the type $ A $ gives us a Segal space completion relative to $ B $ for the Reedy fibrant bisimplicial set $ A $.
\end{observation}

Furthermore, these two notions coincide. Firstly, we introduce notation from  \cite[Proposition 1.2.22]{riehlverity}. The slice $ \bisim/B $ is cotensored over simplicial sets \ie we have $ Map_{/B}(F(n)\times S, X) \simeq Map_{/B}(S, F(n)\pitchfork_B X) $ (here we think of $F(n)$ as a space). Note that since $ X $ is a Segal space over $ B $ so it is $ F(n)\pitchfork_B X $ as is constructed via the pullback
\[\begin{tikzcd}
	{F(n)\pitchfork_B X} & {X^{F(n)}} \\
	B & {B^{F(n)}.}
	\arrow["cst"', from=2-1, to=2-2]
	\arrow[from=1-1, to=2-1]
	\arrow["{\delta^{F(n)}}", from=1-2, to=2-2]
	\arrow[from=1-1, to=1-2]
	\arrow["\lrcorner"{description, pos=0}, draw=none, from=1-1, to=2-2]
\end{tikzcd}\]

\begin{lemma}\label{universaltotuniversal}
	Let $ \pi:A\to B $ and $\xi: S \to B $ in $ \bisim/B  $, where $ B $ is a Segal space. Assume that there is a map $ \iota: A\to S\in \bisim/B$ showing $ S $ as the Segal space completion relative to $ B $ for $ A $. Then for any Segal space $ X $ together with $ \delta:X\to B $ the induced map \[Fun_{/B}(S,X)\to Fun_{/B}(A,X) \] is an equivalence of Segal spaces.
\end{lemma}
\begin{proof}
	Firstly, for any $ n\geq 0 $ we have the following:
	\begin{align*}
		Fun_{/B}(S,X)_n & = Map(F(n),Fun_{/B}(S,X)) \\
		& = Map(F(n),X^S\times_{B^S} F(0)) \\
		& = Map(F(n),X^S)\times_{Map(F(n),B^S)} F(0) \\
		& = Map(F(n)\times S,X)\times_{Map(F(n)\times S,B)} F(0) \\
		& = Map_{/B}(F(n)\times S, X) \\
		& = Map_{/B}(S, F(n)\pitchfork_B X),
	\end{align*}

	We are rely on the fact that $ F(n)\pitchfork_B X $ is a Segal space (see the previous paragraph). Similarly, we get that $ Fun_{/B}(A,X)_n = Map_{/B}(A, F(n)\pitchfork_B X) $ for all $ n\geq 0$.  
	By assumption, $ S\to B $ is the Segal space completion relative to $ B $ for $ A $. Hence, for $ F(n)\pitchfork_B X $ we have an equivalence of spaces
	\[
	Map_{/B}(S, F(n)\pitchfork_B X) \to Map_{/B}(A, F(n)\pitchfork_B X)
	\]
	for all $ n\geq 0 $. This gives us the equivalence between Segal spaces
	\[Fun_{/B}(S,X)\to Fun_{/B}(A,X). \]
\end{proof}

\begin{corollary}
	Given a Segal type $ B $ and any type $ A $ over $ B $. The notion of a relative Segal type completion for $ A $ is consistent with the semantics.
\end{corollary}
\begin{proof}
	This is immediate from \Cref{tuniversaltouniversal} and \Cref{universaltotuniversal}.
\end{proof}

Of course, we also have the non-relative version of this soundness result and it is enough to take $ B $ to be the terminal object. To finalize this section we observe that the fibrant replacement in $ \bisim/B $ coincides with Segal completion relative to $ B $. Since the model structure on $ \bisim/B $ is induced by the one on $ \bisim $ it will be enough to verify this fact for $ B=1 $.

\begin{proposition} \label{completionfibrant}
	Let $ A $ be a Reedy fibrant bisimplicial set. If the Segal space completion of $ A $ exists, then it coincides with its fibrant replacement in the Segal space model structure $ \bisim_{SS}$. 
\end{proposition}
\begin{proof}
	Recall that a map $ i:A \to S $ is a weak equivalence in $ \bisim_{SS} $ if is a local map, \textit{i.e.,} a map that such that
	\[
	i^*:Map(S,X) \to Map(A,X)
	\]
	is equivalence of spaces for any Segal space $ X $. Then it is clear that if $ S $ is a fibrant replacement it must be a Segal space completion. \\
	Conversely, if $ S $ is the Segal space completion then it induces equivalences like the above, so $ i $ is indeed a weak equivalence. $ S $ is a Segal space by assumption, so it must be a fibrant replacement in $ \bisim_{SS} $.
\end{proof}

\subsection{Exponentiable functors} \label{exponentialsemantics}
Here we verify that our notion of exponentiable functors is semantically correct. Ayala, Francis and Rozenblyum prove in \cite{ayala2018factorization} the result below that characterizes exponentiable functors between $ \infty $-categories. This is our reference point.

\begin{theorem}\label{ayalafrancis}
	The following conditions on a functor $ \pi:\catE \to \catB $ between $ \infty $-categories are equivalent.
	\begin{enumerate}
		\item The functor $ \pi $ is an exponentiable fibration.
		\item \label{afcondition3} For each functor $ [2] \to \catB $, the diagram of pullbacks
		\[\begin{tikzcd}
			{\catE \times_{\catB} \{1\}} & {\catE \times_{\catB} \{1<2\}} \\
			{\catE \times_{\catB} \{0<1\}} & {\catE \times_{\catB} [2]}
			\arrow[from=1-1, to=2-1]
			\arrow[from=2-1, to=2-2]
			\arrow[from=1-2, to=2-2]
			\arrow[from=1-1, to=1-2]
		\end{tikzcd}\]
    is a pushout of $ \infty $-categories.
\end{enumerate} 
\end{theorem}

For a full explanation of the theorem we recommend the original reference \cite{ayala2018factorization}. We will focus on Condition \ref{afcondition3} to see this is exactly Condition \ref{conduchecorollary5} of \Cref{conduchecorollary}. This last condition involve objects that are defined by the following pullback square
\[\begin{tikzcd}
	{E_{F(1) \sqcup_{F(0)} F(1)}} & {E_{F(2)}} & E \\
	{F(1) \sqcup_{F(0)} F(1)} & {F(2)} & B.
	\arrow["f", two heads, from=1-3, to=2-3]
	\arrow["\alpha"', from=2-2, to=2-3]
	\arrow[two heads, from=1-2, to=2-2]
	\arrow[from=1-2, to=1-3]
	\arrow[from=1-1, to=2-1]
	\arrow["i"', from=2-1, to=2-2]
	\arrow["\iota", from=1-1, to=1-2]
	\arrow["\lrcorner"{anchor=center, pos=0.125}, draw=none, from=1-1, to=2-2]
	\arrow["\lrcorner"{anchor=center, pos=0.125}, draw=none, from=1-2, to=2-3]
\end{tikzcd}\]
We remark that the arrow on the far letf is not fibration because $F(1) \sqcup_{F(0)} F(1)$ is not a Segal space. Nevertheless, the diagrams expresses the fact that $ E_{F(1) \sqcup_{F(0)} F(1)} $ and $ E_{F(2)} $ are the fibers of $ f $ over $ F(1) \sqcup_{F(0)} F(1) $ and $ F(2)$, respectively.

The map $ E_{F(1) \sqcup_{F(0)} F(1)} \to E_{F(2)} $ shows $ E_{F(2)} $ as the Segal type completion of $ E_{F(1) \sqcup_{F(0)} F(1)} $, thefore when we interpret the square in Condition \ref{conduchecorollary5} from \Cref{conduchecorollary} into $ \bisim_{Segal} $. This gives $ E_{F(2)} $ as the fibrant replacement of $ E_{F(1)}\coprod_{F(0)}E_{F(1)} $ in $ \bisim_{Segal} $. This is just to say that the diagram
\[\begin{tikzcd} \label{pushoutexp}
	{E_{F(0)}} & {E_{F(1)}} \\
	{E_{F(1)}} & {E_{F(2)}}
	\arrow[from=1-1, to=2-1]
	\arrow[from=1-1, to=1-2]
	\arrow[from=1-2, to=2-2]
	\arrow[from=2-1, to=2-2]
\end{tikzcd}\]
is a pushout square of Segal spaces. When $ E $ and $ B $  are Rezk spaces this is exactly Condition \ref{afcondition3} of \Cref{ayalafrancis}. On the other hand, the category of simplicial sets can be embedded into bisimplicial sets via $ p_1^*:\sset \to \bisim $ as defined in \cite{joyaltierney}. Furthermore, this is shown to provide a Quillen adjunction between the Joyal model structure on $ \sset  $ and the complete Segal space model on $ \bisim $. This inclusion preserves exponentials.

\subsubsection{On profunctors and correspondences} \label{profunctors}

Due to the limitations of sHoTT we can not yet incorporate all conditions of \Cref{ayalafrancis} into our \Cref{conduche}. The composition of profunctors appears naturally in Conduch\'e's theorem. The condition \ref{conduchecorollary5} in \Cref{conduchecorollary} carries similar information in the synthetic framework. Given its relevance, in this last section we explain why this is not yet a theorem.

The result in \Cref{ayalafrancis} is expressed and proved using correspondences between $ \infty $-categories. If we have categories $ \catC $ and $ \catD $, a \textbf{correspondence} from $ \catC$ to $\catD $ is category $ \catM $ which contains $\catC $ and $ \catD $ as full subcategories, it is equipped with a functor $ \pi:\catM\to \{0<1\} $ such that $ \catC= \pi^{-1}(0)$ and $ \catD= \pi^{-1}(1)$. While a \textbf{profuctor} from $\catC$ to $\catD$ is a functor $ P:\catC\times\catD^{op}\to \set $. There is a bicategorical equivalence between profuctors from $ \catC $ to $ \catD $ and correspondences from $ \catC $ to $ \catD $. Switching to the realm of $ \infty $-categories:

\begin{definition}
 A \textbf{correspondence} between $ \infty $-categories $ \catC $ and $ \catD $ is a pair of pullbacks:
 \[\begin{tikzcd}
 	\catC & \catM & \catD \\
 	{\{0\}} & {\{0<1\}} & {\{1\}}
 	\arrow[from=1-1, to=2-1]
 	\arrow[from=1-2, to=2-2]
 	\arrow[from=2-1, to=2-2]
 	\arrow[from=2-3, to=2-2]
 	\arrow[from=1-3, to=2-3]
 	\arrow[from=1-3, to=1-2]
 	\arrow[from=1-1, to=1-2]
 \end{tikzcd}\]
This is simply a functor between $ \infty $-categories $ \catM\to \{0<1\} $ with fibers $ \catC $ over $ 0 $ and $ \catD $ over $ 1 $.
\end{definition}

It is a well-known fact a profunctor $ P:\catC\times\catD^{op}\to \set $ can also be defined as a two-sided discrete fibration over $ \catC \times \catD $. This is a functor $ \catE\to \catC \times \catD $ which is a discrete Grothendieck fibration over $ \catD $ and a discrete Grothendieck opfibration over $ \catC $.

Taking into account the limitations of sHoTT, for us it would make sense to momentarily think of profunctors as two-sided discrete fibrations.  Let $ P:A \to B \to \universetyp $ a two-variable type family over Segal types $ A $ and $ B $. From \cite{riehlshulman}, we say that $ P $ is a \textbf{two-sided discrete fibration} if for all $ a:A $ and $ b:B $ the type families
\[ \lambda x.P(x,b):A \to \universetyp \text{ and } \lambda y.P(a,y):B \to \universetyp \]
are contravariant and covariant, respectively. The most famous two-sided discrete fibration over a Segal type $ B $ is the $ ``\homtyp" $ type family
\[ \lambda x.\lambda y.\homtyp_B(x,y): B \to B \to \universetyp. \]
More generally, let $ f:E \to B $ a function between Segal types, $ a, \, b:B $, and $ u:\homtyp_B(a,b) $ then
\[
\lambda x.\lambda y.\homtyp_E^u(x,y): E_a \to E_b \to \universetyp
\]
is a two-sided discrete fibration. The type $ \homtyp_E^u(x,y) $ denotes the type of arrows in $ E $ that start at $ x:E_a $ and end at $ y:E_b $.

Weinberger provides in \cite{weinbergertwo} the following characterization of two-sided discrete families:

\begin{proposition} \label{twosidedcharacterization}
	Given $ P:A \to B \to \universetyp $ a two-side type family over Rezk types, the following are equivalent:
	\begin{enumerate}
		\item The family $ P $ is a two-sided discrete fibration.
		\item The family $ P $ is cartesian over $ A $ and cocartesian over $ B $, and for all $ a:A, \, b:B $ the bifibers $ P(a,b) $ are discrete types. 
	\end{enumerate}
\end{proposition}

We have not introduced cartesian and cocartesian type families in sHoTT, this is the main topic of \cite{buchwein}. However, the meaning of such concepts is in  practice the same one as for $ \infty $-categories. Therefore, the first two conditions of the second point in \Cref{twosidedcharacterization} simply mean that the type families
\[ \lambda x.P(x,b):A \to \universetyp  \text{ and }  \lambda y.P(a,y):B \to \universetyp \]
are cartesian and cocartesian, respectively, for all $ a:A $ and $ b:B $ together with some compatibility condition. This is what \cite{weinbergertwo} defines as two-sided cartesian family. Given $ P:A \to B \to \universetyp $ and $ Q: B \to C \to \universetyp $ two two-sided type families, there is a natural composition to obtain another two-sided type family:
\[
Q \odot P \equiv \lambda a.\lambda c.\sum_{b:B} P(a,b)\times Q(b,c):A \to C \to \universetyp.
\]
The result in \cite[Proposition 5.5]{weinbergertwo} shows that if the families $ P $ and $ Q $ are two-sided cartesian then $ Q\odot P $ is again a two-sided cartesian family. Unfortunately, even if both $ P $ and $ Q $ are two-sided discrete fibrations it does not follow that $ Q\odot P $ is a two-sided discrete fibration. Instead, to make sense of the composition in this case we consider the discrete type completion of $ Q \odot P $.

If we have a function $ f:E \to B $ between Segal types, a condition we would like to add to \Cref{conduche} is the following: For any $ a,\, b,\, c: B, \, u:\homtyp_B(a,b), \, v:\homtyp_B(b,c)$ and $ x:E_a, \, z:E_c $ the canonical map induced by the composition
\[
\left(\sum_{y:E_b} \homtyp_E^u(x,y) \times \homtyp_E^v(y,z)\right) \to \homtyp_E^{v \circ u} (x,z)
\]
exhibits $ \homtyp_E^{v \circ u} (x,z) $ as the discrete type completion of \[ \sum_{y:E_b} \homtyp_E^u(x,y) \times \homtyp_E^v(y,z).\]

The problem arises because Condition \ref{conduchecorollary5} of \Cref{conduchecorollary} encodes the composition of correspondences in sHoTT. These are Segal (Rezk) types over $ \simplext $.  In \cite{stevenson} it is shown that correspondences from $ \catC $ to $ \catD $ are the same as $ \catC\times \catD^{op} \to \mathcal{S} $ where $ \mathcal{S} $ denotes the $ \infty $-category of spaces, and furthermore are the same as a \textit{bifibration}. This is done by endowing the category of correspondences $ Corr(\catC,\catD) $ and the category $ \sset/(\catC\times\catD) $ with model structures, respectively, such that they are Quillen equivalent and where the fibrant objects of $ \sset/(\catC\times\catD) $ are the bifibrations. Both of these models are Quillen equivalent to $ \sset/(\catC \times \catD^{op}) $ with the covariant model structure \textit{i.e.,} this encodes profunctors.

We venture to say that until sHoTT is further enhanced to be more expressive the analogous result from \cite{stevenson} is out of reach. By this we just mean we cannot yet establish a full precise relation between correspondences and two-sided discrete fibrations and profunctors.

	\bibliography{references-exp-functors}
	\bibliographystyle{alpha}
	
\end{document}